\documentclass[11pt]{article}

\usepackage[english]{babel}
\usepackage{enumerate}
\usepackage{xcolor}
\usepackage{authblk}

\usepackage{amsmath}
\usepackage{amsthm}
\usepackage{dsfont}
\usepackage{mathtools}
\usepackage[normalem]{ulem}
\usepackage{tikz}
\usetikzlibrary{shadows}

\usepackage{hyperref}
\hypersetup{
    colorlinks=true,
    linkcolor=blue!75!black,
    urlcolor=gray,
    linktoc=all
}

\xdefinecolor{olive}{cmyk}{0.64,0,0.95,0.4}
\xdefinecolor{blueberry}{cmyk}{0.44,0.23,0.0,0.18}
\xdefinecolor{maroon}{cmyk}{0.00,1.0,1.0,0.498}

\newcommand{\bigzero}{\mbox{\normalfont\Large\bfseries 0}}

\theoremstyle{plain}
\newtheorem{theorem}{Theorem}[section]

\newtheorem{corollary}[theorem]{Corollary}
\newtheorem{lemma}[theorem]{Lemma}

\theoremstyle{definition}

\newtheorem{example}[theorem]{Example}

\newcommand{\ignore}[1]{}

\newcommand{\etal}{{\em et al.~}}

\newcommand{\RR}{\mathds{R}}
\newcommand{\ZZ}{\mathds{Z}}

\newcommand{\QQ}{\mathds{Q}}

\newcommand{\ii}{\mathrm{i}}

\newcommand{\vv}{\mathbf{v}}

\newcommand{\0}{\mathbf{0}}
\newcommand{\1}{\mathbf{1}}

\newcommand{\wt}[1]{\widetilde{#1}}

\newcommand{\md}[1]{\ (\operatorname{mod}   #1)}
\newcommand{\cart}{%
 \ \text{\fboxsep=0.5pt\fbox{\rule[0.5pt]{0pt}{0.75 ex}\rule[0.5pt]{0.75 ex}{0pt}}} \ %
}

\DeclarePairedDelimiter{\abs}{\vert}{\vert}

\DeclareMathOperator{\col}{col}

\DeclareMathOperator{\one}{\mathbf{1}}
\newcommand\pmat[1]{\begin{pmatrix} #1 \end{pmatrix}}
\newcommand\comp[1]{{\mkern2mu\overline{\mkern-2mu#1}}}
\DeclareMathOperator{\ecc}{ecc}

\title{Laplacian Fractional Revival on Graphs}
\author[1]{Ada Chan}

\author[2]{Bobae Johnson}

\author[3]{Mengzhen Liu}

\author[4]{Malena Schmidt}

\author[5]{Zhanghan Yin}

\author[1]{Hanmeng Zhan}

\affil[1]{Department of Mathematics and Statistics, York University}
\affil[2]{Departments of Mathematics and Physics, Harvard University}
\affil[3]{Department of Pure Mathematics and Mathematical Statistics, University of Cambridge}
\affil[4]{Department of Computer Science, University of Warwick  }
\affil[5]{Department of Mathematics, University of Toronto}

\date{}

\begin{document}
\maketitle

\begin{abstract}
We develop the theory of fractional revival in the quantum walk on a graph using its Laplacian matrix as the Hamiltonian. We first give a spectral characterization of Laplacian fractional revival, which leads to a polynomial time algorithm to check this phenomenon and find the earliest time when it occurs. We then apply the characterization theorem to special families of graphs. In particular, we show that no tree admits Laplacian fractional revival except for the paths on two and three vertices, and the only graphs on a prime number of vertices that admit Laplacian fractional revival are double cones. Finally, we construct, through Cartesian products and joins, several infinite families of graphs that admit Laplacian fractional revival; some of these graphs exhibit polygamous fractional revival.
\end{abstract}

\tableofcontents


\section{Introduction}
A quantum walk on a graph $X$ is determined by the transition operator
\[U(t)=\exp(itH),\]
where $H$ is some Hermitian matrix indexed by the vertices. As noted by Bose \etal\cite{bose}, the two common quantum walks, where $H$ is the adjacency or Laplacian matrix of $X$, are related to spin networks in the XY and XYZ interaction models, respectively.

We say $X$ admits \textsl{fractional revival} at time $\tau$ from vertex $a$ to vertex $b$ if 
\[U(\tau)=\alpha e_a + \beta e_b\]
for some complex numbers $\alpha$ and $\beta$. Fractional revival, although rare, is a useful phenomenon in transmitting quantum information \cite{b03}. One of its special cases, perfect state transfer, has been extensively studied from both physical and combinatorial viewpoints. In particular, one can decide adjacency perfect state transfer in polynomial time \cite{PolyTime}, using the spectral information of $A$. There are also characterizations of Laplacian perfect state transfer or adjacency fractional revival on graphs that are combinatorially interesting. For example, Coutinho and Liu \cite{LaPST} ruled out trees on more than two vertices for Laplacian perfect state transfer, and Chan \etal \cite{FRschemes} characterized several families of distance regular graphs with adjacency fractional revival.

Laplacian fractional revival, on the other hand, has not been studied in great detail. There is a characterization of threshold graphs that admit Laplacian fractional revival, due to Kirkland and Zhang \cite{LaFR}, but a general framework, which allows us to investigate this phenomenon systematically, is missing from the literature. 

In this paper, we prove necessary and sufficient conditions (Theorem \ref{CharThm}) for proper Laplacian fractional revival to occur on a general graph. As a consequence, we arrive at a polynomial time algorithm (Theorem \ref{thm:polytime}) that decides Laplacian fractional revival, and if so, finds the earliest time when it occurs. 

We then proceed to determine Laplacian fractional revival on special classes of graphs, including trees (Section \ref{sec:trees}) and join graphs (Section \ref{Section:Joins}). In particular, we prove no tree on more than three vertices admits proper Laplacian fractional revival (Theorem \ref{thm:notree}), which implies the result in \cite{LaPST}. We also show that every graph on a prime number of vertices that admits proper Laplacian fractional revival is a double cone (Theorem \ref{thm:primedbcone}). 

Finally, we use Cartesian products and joins to construct several infinite families of graphs (Theorem \ref{BasicDBCone}, Theorem \ref{thm:infjoin} and Theorem \ref{thm:polygamy}) with proper fractional revival. On the last family, this phenomenon behaves somewhat surprisingly, as it occurs from $a$ to $b$ and from $a$ to $c$ with $b\ne c$. This is in stark contrast to perfect state transfer, for which $a$ must be paired with a unique vertex. As these graphs are regular, they are also examples of unweighted graphs that admit polygamous adjacency fractional revival, which answers an open question in \cite{FRgraphs}.

\section{The Laplacian matrix}
Let $X$ be a graph. Let $A$ be its adjacency matrix, and let $D$ be the diagonal matrix whose $vv$-entry is the degree of the vertex $v$. The \textsl{Laplacian matrix} of $X$, denoted $L$, is given by
\[L = D - A.\]

Our main tool is the spectral decomposition of $L$. In this section, we list some useful facts about $L$; for proofs, see Godsil and Royle \cite[Ch 8]{AGT}.

\begin{lemma}
The Laplacian matrix is positive semidefinite. Moreover, it has $0$ as an eigenvalue with eigenvector $\one$.
\end{lemma}

Next, we cite two upper bounds on the Laplacian eigenvalues due to Anderson and Morley \cite{Levbd}.

\begin{theorem}\cite{Levbd}
\label{Levbd}
Let $X$ be a connected graph on $n$ vertices whose largest Laplacian eigenvalue is $\mu$. Then $\mu\le n$ and
\[\mu\le \max\{\deg(u)+\deg(v): u\sim v\}.\]
Moreover, the multiplicity of $n$ is one less than the number of components of the complement $\comp{X}$.
\end{theorem}

Given an orientation of $X$, the \textsl{signed incidence matrix}, denoted $B$, is the matrix whose rows are indexed by the vertices and columns by the edges with
\[B_{v,e}=\begin{cases}
1,&\text{if $v$ is the head of the edge $e$}\\
-1,& \text{if $v$ is the tail of edge $e$}\\
0,& \text{otherwise}
\end{cases}\]

The following lemma shows a connection between the Laplacian matrix and the signed incidence matrices of a graph.

\begin{lemma}
Let $X$ be a graph with Laplacian matrix $L$. Let $B$ be the signed incidence matrix of any orientation of $X$. Then
\[L = BB^T.\]
\end{lemma}

Finally, we state a well-known result called the Matrix-Tree Theorem. Given a matrix $M$, we let $M[u|v]$ be the matrix obtained from $M$ by deleting the $u$-th row and $v$-th column.

\begin{theorem}[Matrix-Tree Theorem]\label{mattree}
    Let $X$ be a graph on $n$ vertices with Laplacian matrix $L$. Let $q$ be the number of spanning trees of $X$. If $u$ is a vertex of $X$, then $\det(L[u|u])=q$. Moreover, $nq$ equals the product of non-zero eigenvalues of $L$.
\end{theorem}

\section{Laplacian fractional revival}
Given a Hermitian matrix $H$, a quantum walk with respect to $H$ is determined by the following matrix, called the \textsl{transition matrix}:
\[U(t)=\exp(itH).\]
Note that $U(t)$ is a unitary matrix. If, in addition, $H$ is real symmetric, then $U(t)$ is also symmetric.

There are two common choices for $H$ associated with a graph: the adjacency matrix $A$ and the Laplacian matrix $L$. In this paper we will consider the latter. However, for a $d$-regular graph, we have $L=dI-A$, and so 
\[\exp(itL) = e^{idt}\exp(-itA).\]
Therefore, some of our results apply to the adjacency model as well if the graph is regular.

We say a graph admits \textsl{Laplacian fractional revival}, or \textsl{LaFR}, from vertex $a$ to vertex $b$ at time $\tau$ if the transition matrix with respect to $L$ satisfies
\[U(\tau) e_a=\alpha e_a +\beta e_b\]
for some complex numbers $\alpha$ and $\beta$. As $U(t)$ is unitary, it follows that $\abs{\alpha}^2+\abs{\beta}^2=1$. Equivalently, LaFR occurs if $U(\tau)$ has the following block diagonal form:
\begin{equation}\label{def:blockFR}
    U(\tau) =\pmat{
    \begin{matrix}
        \alpha&\beta\\ \beta&\gamma
    \end{matrix}
    &\bigzero\\
    \bigzero&N}.
\end{equation}

There are some special cases of LaFR. If $\beta=0$, then the graph is said to be \textsl{Laplacian periodic} at vertex $a$ at time $\tau$; in this case, vertex $b$ does not play a role. We will refer to the case where $\beta\ne 0$ as \textsl{proper LaFR}. If, on the other hand, $\alpha=0$, then the graph is said to admit \textsl{Laplacian perfect state transfer}, or \textsl{LaPST}, from vertex $a$ to vertex $b$.

We can compute $U(t)$ using the eigenvalues and eigenprojections of $L$. Let the spectral decomposition of $L$ be
\[L=\sum_r \mu_r F_r,\]
where $\mu_r$ is an eigenvalue of $L$, and $F_r$ is the orthogonal projection onto the eigenspace of $\mu_r$. Then we have 
\[U(t)=\sum_r e^{it\mu_r} F_r.\]
This formula will be frequently used in the rest of the paper.

We list some examples that admit LaFR.

\begin{example}\label{ex:K_2}
$K_2$ admits three types of LaFR:
\begin{enumerate}[(i)]
    \item At time $t\in \pi\ZZ$, it is Laplacian periodic at both vertices.
    \item At time $t\in \frac{\pi}{2}\ZZ$, it has LaPST from one vertex to the other.
    \item At time $t\notin \frac{\pi}{2}\ZZ$, it has proper LaFR from one vertex to the other. 
\end{enumerate}
\end{example}

\begin{example}
\ignore{For $P_3$, the Laplacian matrix is 
\[L =  \pmat{1&-1&0\\-1&2&-1\\0&-1&1}\]
and the transition matrix is
\begin{align*}
        U(t) &= \frac{e^{-3it}}{6}\pmat{
            1&-2&1\\
            -2&4&-2\\
            1&-2&1
            }+\frac{e^{-it}}{6}\pmat{
             3&0&-3\\
             0&0&0\\
            -3&0&3
           } +\frac{e^{0}}{6}\pmat{
            2&2&2\\
            2&2&2\\
            2&2&2
            }
    \end{align*}
    We get proper LaFR at time $\frac{2\pi}{3}$ between the two end vertices of $P_3$:
    \begin{align*}
        U(\frac{2\pi}{3}) &= \frac{1}{4}\pmat{
            -\sqrt{3}i+1&0&\sqrt{3}i+3\\
            0&4&0\\
            \sqrt{3}i+3&0&-\sqrt{3}i+1
            }
    \end{align*}
    and periodicity at time $2\pi$ at all vertices:
    \begin{align*}
        U(2\pi) &= \pmat{
            1&0&0\\
            0&1&0\\
           0&0&1
            }
    \end{align*}}
On $P_3$, there is proper LaFR at time $2\pi/3$ between the end vertices, and periodicity at time $2\pi$ at all vertices.

\end{example}

The next example takes a look at four graphs that admit proper LaFR in a very interesting way. 

\begin{example}\label{4egs}
In Figure \ref{fig4egs}, the vertices that are not black in each graph are the ones between which proper LaFR occurs. For example, in the first graph, $C_6$, proper LaFR occurs between antipodal vertices. Comparing the first to the second graph, we can see that adding an edge between two vertices that do not have LaFR destroys their LaFR with their matching vertices. Going from the second to the fourth graph still preserves the LaFR. Now, what is particularly interesting is the comparison between the second and the third graph: Adding an edge between the two vertices admitting proper LaFR still preserves the proper LaFR. All of these are just observations on these specific graphs and do not apply to all graphs. 
\end{example}

\begin{figure}[h]\label{fig4egs}
\centering
\begin{tikzpicture}
\path (1,0) coordinate (a1);
\fill [blue]  (a1)  circle (2pt);
\path  (2,1) coordinate (a2);
\fill [green] (a2)  circle (2pt);
\path (2,2) coordinate (a3);
\fill [red] (a3)  circle (2pt);
\path (1,3) coordinate (a4);
\fill [blue]  (a4)  circle (2pt);
\path (0,2) coordinate (a5);
\fill [green]  (a5)  circle (2pt);
\path (0,1) coordinate (a6);
\fill [red] (a6)  circle (2pt);
\draw (a1)--(a2)--(a3)--(a4)--(a5)--(a6)--(a1);

\path (3.5,0) coordinate (b1);
\fill [blue]  (b1)  circle (2pt);
\path (4.5,1) coordinate (b2);
\fill  (b2)  circle (2pt);
\path (4.5,2) coordinate (b3);
\fill  (b3)  circle (2pt);
\path (3.5,3) coordinate (b4);
\fill [blue]  (b4)  circle (2pt);
\path (2.5,2) coordinate (b5);
\fill  (b5)  circle (2pt);
\path (2.5,1) coordinate (b6);
\fill  (b6)  circle (2pt);
\draw (b1)--(b2)--(b3)--(b4)--(b5)--(b6)--(b1);
\draw (b2)--(b6);

\path (6,0) coordinate (c1);
\fill [blue]  (c1)  circle (2pt);
\path (7,1) coordinate (c2);
\fill  (c2)  circle (2pt);
\path (7,2) coordinate (c3);
\fill  (c3)  circle (2pt);
\path (6,3) coordinate (c4);
\fill [blue]  (c4)  circle (2pt);
\path (5,2) coordinate (c5);
\fill  (c5)  circle (2pt);
\path (5,1) coordinate (c6);
\fill  (c6)  circle (2pt);
\draw (c1)--(c2)--(c3)--(c4)--(c5)--(c6)--(c1);
\draw (c2)--(c6);
\draw (c1)--(c4);

\path (8.5,0) coordinate (d1);
\fill [blue]  (d1)  circle (2pt);
\path (9.5,1) coordinate (d2);
\fill  (d2)  circle (2pt);
\path (9.5,2) coordinate (d3);
\fill  (d3)  circle (2pt);
\path (8.5,3) coordinate (d4);
\fill [blue]  (d4)  circle (2pt);
\path (7.5,2) coordinate (d5);
\fill  (d5)  circle (2pt);
\path (7.5,1) coordinate (d6);
\fill  (d6)  circle (2pt);
\draw (d1)--(d2)--(d3)--(d4)--(d5)--(d6)--(d1);
\draw (d2)--(d6);
\draw (d3)--(d5);
\end{tikzpicture}
\caption{Example \ref{4egs}}
\end{figure}

\section{Cartesian products}
In this section, we discuss a graph operation that preserves LaFR.

Let $X$ and $Y$ be two graphs. The \textsl{Cartesian product} of $X$ and $Y$, denoted $X\cart Y$, is the graph with vertex set $V(X) \times V(Y)$, such that two vertices $(x_1, y_1)$ and
$(x_2, y_2)$ are adjacent if either $x_1$ is adjacent to $x_2$ in $X$ and $y_1=y_2$, or $x_1=x_2$ and $y_1$ is adjacent to $y_2$. Using the Kronecker product of matrices, we can express the Laplacian matrix of $X\cart Y$ in terms of the Laplacian matrices of $X$ and $Y$:
\[L(X\cart Y) = L(X)\otimes I + I\otimes L(Y);\]
for details about this identity, see Godsil and Coutinho \cite{GSCQW}. As $L(X)\otimes I$ commutes with $I\otimes L(Y)$, it follows that
\[U_{X\cart Y} (t) = U_X(t) \otimes U_Y(t).\]

We characterize proper LaFR on Cartesian products in terms of their base graphs.
\begin{theorem}[Proper LaFR in Cartesian Products]\label{thm:cartproduct}
    Let $X$ and $Y$ be two graphs. The Cartesian product $X\cart Y$ has proper LaFR at time $\tau$ if and only if $X$ has a periodic vertex at time $\tau$ and $Y$ has proper LaFR at time $\tau$.
\end{theorem}
\begin{proof}
    First, suppose $X \cart Y$ has proper LaFR at time $\tau$. Then
    \begin{align*}
    U_{X\cart Y}(\tau) &= U_X(\tau)\otimes U_Y(\tau) = \pmat{
    \begin{matrix}
        \alpha&\beta\\ \beta&\gamma
    \end{matrix}
    &\bigzero\\
    \bigzero&N}
    \end{align*}
    Let $a_{ij}$ and $b_{ij}$ be the $ij$-entries of $U_X(\tau)$ and $U_Y(\tau)$, respectively. We have
\begin{equation}
\label{Eqn:top}
    a_{11}b_{11} = \alpha, \quad
    a_{11}b_{21} = \beta, \quad
    a_{11}b_{12} = \beta, \quad
    a_{11}b_{22} = \gamma,
\end{equation}
and for $j\ge 3$,
\begin{equation}
\label{Eqn:zeros}
    a_{11}b_{j1} = 0, \quad
    a_{11}b_{j2} = 0, \quad
    a_{11}b_{1j} = 0, \quad 
    a_{11}b_{2j} = 0.
\end{equation}
As LaFR on $X\cart Y$ is proper, $\beta\ne0$, and so $a_{11}\ne0$. From \eqref{Eqn:top} it follows that $a_{11}\neq 0$, $b_{12}\neq 0$ and $b_{21}\neq 0$. Further, from \eqref{Eqn:zeros} we can conclude that for $j\ge 3$,
\begin{equation}
\label{Eqn:zerosb}
    b_{j1} = 0, \quad
    b_{j2} = 0, \quad
    b_{1j} = 0, \quad
    b_{2j} = 0.
\end{equation}
So, Y indeed has proper LaFR at time $\tau$. Now, since $b_{12}\neq 0$ and $b_{21}\neq 0$, by similar reasoning we have $a_{j1} = a_{1j} = 0$ for all $j \neq 1$.Therefore, $X$ must have a periodic vertex at $\tau$. 

Conversely, suppose at time $\tau$, $X$ has a periodic vertex and $Y$ has proper LaFR. Then 
\begin{align*}
    U_X(\tau) &= 
    \begin{pmatrix}
    1&0&\dots&0\\
    \begin{matrix}
    0\\ \vdots\\ 0
    \end{matrix}&&\begin{matrix}
    N_x
    \end{matrix}
    \end{pmatrix}
\end{align*}
for some $(n-1)\times(n-1)$ matrix $N_x$, and
\begin{align*}
    U_Y(\tau) &= 
    \begin{pmatrix}
        \begin{matrix}
            N
        \end{matrix}
        & \bigzero \\
        \bigzero  &
        \begin{matrix}
            N_y
        \end{matrix}
    \end{pmatrix}
    =
    \begin{pmatrix}
        \begin{matrix}
            \alpha & \beta \\
            \beta & \gamma
        \end{matrix}
        & 0 \\
        0  &
        \begin{matrix}
            N_y
        \end{matrix}
    \end{pmatrix}
\end{align*}
for some $(n-2)\times(n-2)$ matrix $N_y$. It follows that
\begin{align*}
    U_X(\tau) \otimes U_Y(\tau) &=
    \begin{pmatrix}
        \begin{matrix}
            \alpha & \beta \\
            \beta & \gamma
        \end{matrix}&\bigzero\\
        \bigzero&
        \begin{matrix}
            N_x\otimes N_y
        \end{matrix}
    \end{pmatrix}
\end{align*}
As LaFR on $Y$ is proper, $\beta \ne 0$, and so LaFR on $X\cart Y$ is proper.
\end{proof}

\section{Complements and joins}
\label{sec:compjoin}
Under certain conditions, LaFR is also preserved by taking complements or joins. The following result generalizes Theorem 2 of \cite{PSTLa}.

\begin{theorem}\label{thm:complement}
Let $X$ be a graph on $n$ vertices. If $X$ has LaPST or LaFR between two vertices $u$ and $v$ at time $\tau$ , with $n\tau \in 2\pi \ZZ$, then the complement $\comp{X}$ has LaPST or LaFR, respectively, between these two vertices at $\tau$.
\end{theorem}
\begin{proof}
Let $L$ be the Laplacian matrix of $X$, and $\comp{L}$ the Laplacian matrix of $\comp{X}$. We have
\[\comp{L} = nI-J-L.\]
Since $L$ commutes with $J$,
\begin{align*}
\exp(it\comp{L})
&=e^{int}\exp(-itJ)\exp(-itL)\\
&=e^{int}\left((e^{-int}-1)\frac{1}{n}J+I\right)\exp(-itL)
\end{align*}
Thus, if $nt \in 2\pi\ZZ$, we have $\exp(it\comp{L})=\exp(-itL)$.
\end{proof}

Given two graphs $X$ and $Y$, the \textsl{join} of $X$ and $Y$, denoted $X+Y$, is the graph obtained from the disjoint union $X\cup Y$ by joining all vertices of $X$ to all vertices of $Y$. Equivalently,
\[X+Y = \comp{\comp{X}\cup \comp{Y}}.\]

\begin{corollary}\label{cor:join}
Let $X$ be a graph whose complement admits LaPST or LaFR between vertex $a$ and vertex $b$ at time $\tau$. For any graph $Y$, the join $X+Y$ has LaPST or LaFR, respectively, between these two vertices at time $\tau$ provided that $\tau(\abs{V(X)}+\abs{V(Y)})\in 2\pi\ZZ$.
\end{corollary}
\begin{proof}
This follows from the above identity and Theorem \ref{thm:complement}.
\end{proof}

As a special case of joins, the graph $\comp{K_2}+Y$ is called the \textsl{double cone} over $Y$. We will refer to the two vertices from $K_2$ in a double cone as the \textsl{conical vertices}. Our next result shows that double cones provide an infinite family of graphs on which proper LaFR occurs.

\begin{theorem}\label{BasicDBCone}
    Let $n\ge 3$ be an integer. Then any double cone on $n$ vertices admits proper LaFR at time $2\pi/n$ between the conical vertices. Moreover, the LaFR is LaPST if and only if $n=4$.
\end{theorem}
\begin{proof}
Let $\comp{K_2}+Y$ be a double cone on $n$ vertices, and let $\tau=2\pi/n$. As $n\ge 3$, we have $\tau \notin \pi\ZZ$, and so by Example \ref{ex:K_2}, $K_2$ has proper LaFR at time $\tau$. Now apply Corollary \ref{cor:join} with $X = \comp{K_2}$. 

Since LaPST occurs on $K_2$ at odd multiples of $\pi/2$, the double cone $\comp{K_2}+Y$ admits LaPST at time $2\pi/n$ if and only if $n = 4$.
\end{proof}

At the end of this section, we mention another infinite family of graphs, found by Kirkland and Zhang \cite{LaFR}, with LaFR.

Let $O_n=\comp{K_n}$. The \textsl{threshold graph}, denoted $\Gamma(m_1, m_2, \ldots, m_{2k})$, is the graph 
\begin{equation*}
    \Bigg(\bigg( \Big( \big((O_{m_1}+K_{m_2})\cup O_{m_3} \big)+K_{m_4} \Big)\cup \cdots\bigg) O_{m_{2k-1}} \Bigg)+K_{m_{2k}}
\end{equation*}

\begin{theorem}\cite{LaFR}
The threshold graph $\Gamma(m_1, m_2, \ldots, m_{2k})$ admits
LaFR between vertex $a$ and vertex $b$ at time $\tau$ if and only if
\begin{enumerate}[(i)]
\item
$m_1=2$ and $V(O_{m_1})=\{a,b\}$,
\item
$\tau$ is not an integer multiple of $\pi/2$, and
\item
$(m_1+m_2)\tau\equiv 0 \md{2\pi}$, and $m_j\tau \equiv 0 \md{2\pi}$, for $j=3,\ldots,2k$.
\end{enumerate}
\end{theorem}

\section{Laplacian strong cospectrality}
In this section, we study a spectral property called Laplacian strong cospectrality. This is a key property required by proper LaFR.

Let $X$ be a graph. Let $L$ be its Laplacian matrix with spectral decomposition
\[L=\sum_r \mu_r F_r.\]
We say two vertices $a$ and $b$ are \textsl{Laplacian strongly cospectral} if for each $r$, 
\[F_r e_a = \pm F_r e_b.\]
This extends the notion of adjacency strong cospectrality due to Godsil and Smith \cite{SCV}. It is proved that if a graph admits perfect state transfer between two vertices $a$ and $b$, then they must be strongly cospectral. In the next section, we will show that Laplacian strong cospectrality is also a necessary condition for LaFR.

Given two strongly cospectral vertices $a$ and $b$ with respect to $L$, there is a natural partition of the eigenvalues of $L$:
\begin{eqnarray*}
\Phi_{ab}^+ &=& \{\mu_r: F_re_a = F_r e_b \ne 0\}\\
\Phi_{ab}^- &=& \{\mu_r: F_re_a = -F_r e_b \ne 0\}\\
\Phi_{ab}^0 &=& \{\mu_r: F_re_a=F_re_b=0\}
\end{eqnarray*}
In all three sets, we have 
\[(F_r)_{aa} = e_a^T F_r e_a = e_b^T F_r e_b = (F_r)_{bb}.\]
This leads to the following observation.

\begin{lemma}\label{lem:samedeg}
Let $a$ and $b$ be two Laplacian strongly cospectral vertices. Then $\deg(a)=\deg(b)$.
\end{lemma}
\begin{proof}
Let the spectral decomposition of $L$ be
\[L = \sum_r \mu_r F_r.\]
Then
\[L_{aa} = \sum_r \mu_r (F_r)_{aa} =\sum_r \mu_r(F_r)_{bb}.\]
The result follows as the diagonal entries of $L$ are degrees of the vertices.
\end{proof}

Our next lemma gives the structure of the spectral idempotents summed over each class.

\begin{lemma}\label{lem:keylemma}
Let $X$ be a graph. Let $L$ be its Laplacian matrix with spectral decomposition
\[L=\sum_r \mu_r F_r.\]
Suppose vertices $a$ and $b$ are strongly cospectral with respect to $L$. Then 
\[
        \sum_{r:\mu_r\in \Phi_{ab}^+}F_r= \frac12\begin{pmatrix}
            \begin{matrix}
                1&1\\
                1&1
            \end{matrix}
            &0\\
            0&*
        \end{pmatrix},\quad
    \sum_{r:\mu_r\in \Phi_{ab}^-}F_r= \frac12\begin{pmatrix}
            \begin{matrix}
                1&-1\\
                -1&1
            \end{matrix}
            &0\\
            0&*'
        \end{pmatrix}.
\]
\end{lemma}
\begin{proof}
Define
\[v^+=\sum_{r:\mu_r\in \Phi_{ab}^+}F_re_a,\quad v^-=\sum_{r:\mu_r\in \Phi_{ab}^-}F_re_a.\]
As the spectral idempotents $F_r$ sum to the identity, we have $e_a=v^++v^-$ and $e_b=v^+-v^-$. Hence
$$v^+=\frac 12(e_a+e_b),\quad v^-=\frac 12(e_a-e_b)$$
which corresponds to the first columns of $\displaystyle\sum_{r:\mu_r\in \Phi_{ab}^+}F_r$ and $\displaystyle\sum_{r:\mu_r\in \Phi_{ab}^-}F_r$. The forms of their second column and the rest of their first two rows follows from strong cospectrality and that the matrices are symmetric.
\end{proof}

If $a$ and $b$ are Laplacian strongly cospectral, then the eigenvalue $0$ lies in $\Phi_{ab}$. By the two identities in Lemma \ref{lem:keylemma}, $\Phi_{ab}^+$ contains at least one non-zero eigenvalue, and $\Phi_{ab}^-$ cannot be empty. Hence we have the following lower bounds.

\begin{corollary}\label{pmsetsize}
Let $X$ be a graph with at least three vertices. If $a,b$ are Laplacian strongly cospectral vertices in $X$, then $|\Phi_{ab}^+|\ge 2$ and $|\Phi_{ab}^-|\ge 1$.
\end{corollary}

We can say more about this partition.

\begin{theorem}\label{algclosed}
 Suppose $a,b$ are strongly cospectral vertices. Then $\Phi_{ab}^+,\Phi_{ab}^-,\Phi_{ab}^0$ are individually closed under algebraic conjugation.
\end{theorem}
\begin{proof}
Let $\mu$ be an eigenvalue which is not an integer and $\vv$ any of its eigenvectors. Let $\mu'$ be an algebraic conjugate of $\mu$. There is a field automorphism $\Psi$ of
$\QQ(\mu)$ that maps $\mu$ to $\mu'$. It also maps $\vv$ to a vector $\vv'$, which is an eigenvector for the eigenvalue $\mu'$. As $a$ and $b$ are strongly cospectral, it follows that
\[\vv_a = \pm \vv_b.\]
Now since $\Psi(1) = 1$ and $\Psi(-1) = -1$, we have
\[\vv'_a = \pm \vv'_b\]
with the same sign. Hence $\Phi_{ab}^+,\Phi_{ab}^-,\Phi_{ab}^0$ are closed under $\Psi$.
\end{proof}

We have seen that Laplacian strongly cospectral vertices have the same degree. In the following theorem, we give an upper bound and lower bound for this degree.

\begin{theorem}\label{scdegbound}
Let $X$ be a connected graph on $n$ vertices. Let $a$ and $b$ be two strongly cospectral vertices with respect to the Laplacian matrix. Let $\theta^{\pm}$ denote the largest element in $\Phi_{ab}^{\pm}$, and $\lambda^{\pm}$ the smallest non-zero element in $\Phi_{ab}^{\pm}$. Then the following hold.
\begin{enumerate}[(i)]
    \item If $a$ and $b$ are not adjacent, then 
    \[\max\left\{\frac{n-2}{n}\lambda^+, \lambda^-\right\} \le \deg(a) \le \min\left\{\frac{n-2}{n}\theta^+, \theta^-\right\}\]
    \item If $a$ and $b$ are adjacent, then 
    \[\max\left\{\frac{n-2}{n}\lambda^+ +1, \lambda^- -1\right\} \le \deg(a) \le \min\left\{\frac{n-2}{n}\theta^+ +1, \theta^- -1\right\}\]
\end{enumerate}
\end{theorem}

\begin{proof}
Given a matrix $M$, let $\wt{M}$ denote the $\{a,b\}$ principal submatrix of $M$. As $a$ and $b$ are Laplacian strongly cospectral, for each eigenvalue $\mu_r\in \Phi_{ab}^+$, there is a non-zero constant $c_r$ such that
\[\wt{F_r} = c_r \pmat{1&1\\1&1},\]
and for each eigenvalue $\mu_r\in \Phi_{ab}^-$, there is a non-zero constant $c_r$ such that
\[\wt{F_r} = c_r \pmat{1&-1\\-1&1}.\]

On the other hand, if $d=\deg(a)$ then
\[\wt{L}=\pmat{d & -\sigma\\-\sigma & d},\]
where $\sigma=1$ if $a$ and $b$ are adjacent, and $\sigma=0$ otherwise. We abuse notation and denote $\displaystyle r\colon\mu_r\in\Phi_{ab}^\pm$ by $r\in\Phi_{ab}^\pm$ for convenience. Hence by 
\[\sum_{r\in \Phi_{ab}^+} \mu_r F_r + \sum_{r\in \Phi_{ab}^-} \mu_r F_r= L,\]
we obtain two equalities:
\begin{align*}
    \sum_{r\in\Phi_{ab}^+} \mu_r c_r +\sum_{r\in\Phi_{ab}^-} \mu_r c_r&=d,\\
    \sum_{r\in\Phi_{ab}^+} \mu_r c_r -\sum_{r\in\Phi_{ab}^-} \mu_r c_r&=-\sigma.
\end{align*}
Thus
\begin{align}
    \sum_{r\in \Phi_{ab}^+} \mu_r c_r &=\frac{d-\sigma}{2}, \label{eqn:+}\\
    \sum_{r\in \Phi_{ab}^-} \mu_r c_r &=\frac{d+\sigma}{2}. \label{eqn:-}
\end{align}
Now, recall that $0$ is an eigenvalue in $\Phi_{ab}^+$ with eigenvector $\one$. Moreover, by Lemma \ref{lem:keylemma},
\[\sum_{r\in\Phi_{ab}^+}c_r =\sum_{r\in\Phi_{ab}^-}c_r=\frac{1}{2}.\]
Therefore Equation \ref{eqn:+} tells us that
\[
    \frac{d-\sigma}{2}
    =\sum_{r\in\Phi_{ab}^+\backslash\{0\}} \mu_r c_r
    \ge \sum_{r\in\Phi_{ab}^+\backslash\{0\}} \lambda^+ c_r
    =\lambda^+\left(\sum_{r\in\Phi_{ab}^+\setminus\{0\}} c_r+\frac{1}{n}\right)-\frac{1}{n}\lambda^+
    =\left(\frac{1}{2}-\frac{1}{n}\right)\lambda^+.
\]

It follows that
\[d\ge \frac{n-2}{n}\lambda^+ +\sigma.\]
Likewise, Equation \eqref{eqn:-} tells us that
\[\frac{d+\sigma}{2}\ge \sum_{r\in\Phi_{ab}^-} \lambda^- c_r = \frac{1}{2}\lambda^-.\]
Hence
\[d\ge \lambda^- -\sigma.\]
The upper bound for $d$ follows from a similar argument.
\end{proof}



Our next result links the integer eigenvalues in $\Phi_{ab}^-$ to the number of spanning trees of the graph.

\begin{lemma}\label{lem:minussetfactor}
Let $X$ be a graph on $n$ vertices with Laplacian strongly cospectral vertices $a$ and $b$. Let $\mu\in \Phi_{ab}^-$ be an integer. Then any odd prime $p$ that divides $\mu$ must divide the number of spanning trees of $X$.
\end{lemma}
\begin{proof}
Let $q$ be the number of spanning trees of $X$. By the Matrix-Tree Theorem, any $(n-1)\times(n-1)$ minor of the Laplacian matrix $L$ equals $q$. Therefore, for each odd prime $p$ not dividing $q$, the Laplacian matrix has rank $n-1$ over $\ZZ_p$, with kernel being the span of $\one$.

Now let $a$ and $b$ be two strongly cospectral vertices with respect to $L$. Suppose, for a contradiction, that some integer eigenvalue $\mu\in\Phi_{ab}^-$ has an odd prime factor $p$ that does not divide $q$. Let $x$ be an eigenvector of $\mu$, and assume without loss of generality that the gcd of its entries is $1$. Then since
\[Lx\equiv 0\pmod{p},\]
there is an integer $k$ such that
\[x\equiv k\one \pmod{p}.\]
As $\mu\in\Phi_{ab}^-$, the projection of $e_a$ and $e_b$ onto any subspace of the $\mu$-eigenspace must be opposites. Therefore $x_a=-x_b$. Hence $k\equiv -k \pmod{p}$. But as $p$ is an odd prime, $k$ is divisible by $p$. Hence $p$ divides all entries of $x$. This contradicts the assumption that all entries of $x$ have $\gcd$ $1$.
\end{proof}

Finally, we show a connection between the size of $\Phi_{ab}^-$ and the distance between $a$ and $b$.

\begin{theorem}\label{thm:scdist}
Let $a$ and $b$ be Laplacian strongly cospectral vertices in $X$. If $\abs{\Phi_{ab}^-}=k$, then any vertex at distance $k$ from $a$ is at distance at most $k$ from $b$. In particular, the distance between $a$ and $b$ is no more than $2k$. 
\end{theorem}
\begin{proof}
Define as in Lemma \ref{lem:keylemma}
\[v^-:=\sum_{r\in \Phi_{ab}^-}F_re_a=\frac 12(e_a-e_b),\]
and set
\[W=\prod_{r\in\Phi_{ab}^-}(L-\mu_r I).\]
Then we have $Wv^-=0$. Hence $We_a=We_b$.

Now, notice that $W$ is the product of $k$ matrices in the form $(D_i-A)$, where $A$ is the adjacency matrix and each $D_i$ is some diagonal matrix. 
By the expanded form of $W$, the $ij$-entry is non-zero only if $d(i,j)\le k$, and is $\pm1$ if $d(i,j)=k$. Since the $a$-th column and the $b$-th column of $W$ are identical, we conclude that any vertex at distance $k$ from $a$ is at distance at most $k$ from $b$.
\end{proof}

As a consequence, we have proved the following result due to Coutinho and Liu \cite{LaPST}. Two vertices are called \textsl{twins} if they have the same neighbors.

\begin{corollary}\cite{LaPST}\label{scdist}
Let $a$ and $b$ be Laplacian strongly cospectral vertices in $X$. If $\abs{\Phi_{ab}^-}=1$, then $a$ and $b$ are twin vertices. We allow twins to be adjacent.
\end{corollary}

\section{Proper Laplacian fractional revival}
We now characterize proper LaFR on connected graphs. The following lemma shows that $\alpha=\gamma$ in Equation \eqref{def:blockFR}, provided $\beta\ne 0$. This is an important observation as it leads to Laplacian strong cospectrality.

\begin{lemma}\label{lem:cospectral}
If $X$ admits proper LaFR between $a$ and $b$ at time $\tau$, then $U(\tau)_{aa}=U(\tau)_{bb}$. Moreover, each spectral idempotent $F_r$ of $L$ satisfies $(F_r)_{aa}=(F_r)_{bb}$.
\end{lemma}
\begin{proof}
We already know that
\[U(\tau) = \pmat{
    \begin{matrix}
        \alpha&\beta\\ \beta&\gamma
    \end{matrix}
    &\bigzero\\
    \bigzero&N}.\]
Let $M$ be the submatrix of $U(\tau)$ indexed by $a$ and $b$. Since $\beta \ne 0$, $M$ has two distinct eigenvalues. Moreover, as $\one$ is an eigenvector of $L$, the eigenvectors of $M$ must be
\[\pmat{1\\1},\quad \pmat{1\\-1}.\]
Therefore $\alpha=\gamma$.

For the second statement, let the spectral decomposition of $L$ be
\[L = \sum_r \mu_r F_r.\] 
Then
\[\sum_r e^{it\mu_r} F_r= \pmat{M & \0 \\ \0& N}.\]
Multiply both sides by $F_r$ and we get
\[e^{it\mu_r} F_r = \pmat{M & \0 \\ \0& N}F_r.\]
As before, let $\wt{F_r}$ denote the submatrix of $F_r$ indexed by $a$ and $b$. We have
\[e^{it\mu_r} \wt{F_r} = M \wt{F_r},\]
from which it follows that the column space of $\wt{F_r}$ is a subspace of an eigenspace of $M$. Therefore, $(F_r)_{aa}=(F_r)_{bb}$.
\end{proof}

We are now ready to show that proper LaFR can only happen between Laplacian strongly cospectral vertices.

\begin{theorem}\label{LaFRSC}
If $X$ admits proper LaFR between $a$ and $b$, then $a$ and $b$ are Laplacian strongly cospectral.
\end{theorem}
\begin{proof}
Let $L$ the Laplacian matrix of $X$ with spectral decomposition
\[L = \sum_r \mu_r F_r.\]
Suppose proper LaFR occurs between $a$ and $b$ at time $\tau$. Then there are complex numbers $\alpha$ and $\beta$, where $\beta\ne 0$, such that
\[\sum_r e^{it\mu_r} F_re_a =\alpha e_a + \beta e_b.\]
For each $r$, multiplying both sides by $F_r$ yields
\[e^{it\mu_r} F_r e_a = \alpha F_r e_a + \beta F_r e_b.\]
As $\beta$ is non-zero, $F_r e_b$ is a scalar multiple of $F_r e_a$. Finally, Lemma \ref{lem:cospectral} tells us that $(F_r)_{aa}=(F_r)_{bb}$. Hence it must be that $F_r e_a =\pm F_r e_b$.
\end{proof}

Our next result shows that each eigenvalue in $\Phi_{ab}^+$ or $\Phi_{ab}^-$ respects the class it belongs to. It becomes useful when we try to determine the time when proper LaFR occurs.

\begin{lemma}\label{constant}
Suppose proper LaFR occurs between $a$ and $b$ at time $\tau$. Then the function $\mu_r\mapsto e^{i\tau\mu_r}$ is constant within each of $\Phi_{ab}^+$ and $\Phi_{ab}^-$.
\end{lemma}
\begin{proof}
We use a similar argument to the proof of Lemma \ref{lem:cospectral}. 
Let the spectral decomposition of $L$ be
\[L = \sum_r \mu_r F_r.\] 
Assuming the first two rows of $L$ are indexed by $a$ and $b$, we have
\[\sum_r e^{it\mu_r} F_r= \pmat{M & \0 \\ \0& N},\]
Since the LaFR proper, $M$ is non-diagonal, and so it has two distinct eigenvalues with eigenvectors
\[\pmat{1\\1},\quad \pmat{1\\-1},\]
respectively. Now for each $r$,
\[e^{it\mu_r} F_r = \pmat{M & \0 \\ \0& N}F_r.\]
If $\wt{F_r}$ denotes the submatrix of $F_r$ indexed by $a$ and $b$, then
\[e^{it\mu_r} \wt{F_r} = M \wt{F_r},\]
that is, the column space of $\wt{F_r}$ is a subspace of an eigenspace of $M$. On the other hand, we know from the proof of Theorem \ref{scdegbound} that each $\wt{F_r}$ is rank-one; more specifically, if $\mu_r\in \Phi_{ab}^+$,
\[
    \wt{F_r} =
    c_r \pmat{1&1\\1&1}\]
and if $\mu_r\in \Phi_{ab}^-$.
\[\wt{F_r}=c_r \pmat{1&-1\\-1&1}.\]
Thus, if $\mu_r$ and $\mu_s$ both lie in $\Phi_{ab}^+$ or both lie in $\Phi_{ab}^-$, then $\wt{F_r}$ and $\wt{F_s}$ are scalar multiples of each other, and so $e^{i\tau\mu_r}= e^{i\tau\mu_s}$.
\end{proof}

As a consequence, we get a ratio condition on the eigenvalues when proper LaFR occurs.

\begin{corollary}\label{ratiocondition}

Suppose $X$ admits proper LaFR between $a$ and $b$. Then 
\begin{equation*}
\frac{\mu_i-\mu_j}{\mu_r-\mu_s} \in \QQ
\end{equation*}
for all  $\mu_i, \mu_j, \mu_r, \mu_s \in \Phi_{a,b}^+$, or  for all  $\mu_i, \mu_j, \mu_r, \mu_s \in \Phi_{a,b}^-$, with $\mu_r\neq \mu_s$.
\end{corollary}
\begin{proof}
For $\mu_i,\mu_j$ in the same $\Phi_{ab}^+$ and $\Phi_{ab}^-$ class, $e^{i\tau\mu_i}=e^{i\tau\mu_j}$. Hence \[\tau(\mu_i-\mu_j)\in 2\pi\ZZ.\]
The result follows by taking the ratio of $\tau(\mu_i-\mu_j)$ and $\tau(\mu_r-\mu_s)$.
\end{proof}

We cite a powerful result due to Godsil \cite[Theorem 6.1]{g12b}; it bounds the algebraic degrees of the elements in $\Phi_{ab}^+$ and $\Phi_{ab}^-$.

\begin{theorem}\cite{g12b}\label{harmonious}
Let $\Phi$ be a set of real algebraic integers which is closed under taking algebraic conjugates. Suppose for all $\mu_i,\mu_j,\mu_r,\mu_s \in \Phi$ with $\mu_r\ne \mu_s$, we have
\[
\frac{\mu_i - \mu_j}{\mu_r - \mu_s} \in \QQ.
\]
Then the elements of $\Phi$ are either integers or quadratic integers, and, moreover, if $|\Phi| = n$, then there are integers $a$, $\Delta$ (square-free), and $\{b_r\}_{r = 1}^n$ such that $\mu \in \Phi$ implies that, for some $r$,
\[\mu = \frac{a + b_r \sqrt{\Delta}}{2}.\] 
\end{theorem}

With all the theory we developed so far, we are now ready to give a characterization of proper LaFR.

\begin{theorem}\label{CharThm}
(Characterization Theorem for proper LaFR) Let $X$ be a graph on at least three vertices. For vertices $a$ and $b$, let
\[g = \gcd \left\{\mu_r-\mu_s:\mu_r, \mu_s \in \Phi_{ab}^+ \text{ or } \mu_r, \mu_s\Phi_{ab}^- \right\}.\]
proper LaFR occurs between $a$ and $b$ if and only if the following conditions hold:
\begin{enumerate}[(i)]
    \item $a$ and $b$ are strongly cospectral vertices.
    \item $\Phi_{ab}^+\cup \Phi_{ab}^-\subseteq \ZZ^{\ge 0}$
    \item There is some element in $\Phi_{ab}^-$ that is not divisible by $g$.
\end{enumerate}
Moreover, if proper LaFR occurs between $a$ and $b$ at time $\tau$, then $\tau$ is an integer multiple of $2\pi/g$.
\end{theorem}
\begin{proof}
We first prove the if direction. Let $\tau=2\pi/g$. By the definition of $g$, the function $\mu_r\mapsto e^{i\tau \mu_r}$ is constant within $\Phi_{ab}^+$ and within $\Phi_{ab}^-$. Let $\mu_s$ be an element in $\Phi_{ab}^-$ that is not divisible by $g$. Then $e^{i\tau\mu_s}\ne 1$. Now, noticing $0\in \Phi_{ab}^+$ and using Lemma \ref{lem:keylemma},
\begin{align*}
U(\tau) &= e^{i\tau\cdot 0} \left(\sum_{u:\mu_u\in \Phi_{ab}^+}F_u\right)+e^{i\tau\mu_s}\displaystyle\left(\sum_{u:\mu_u\in \Phi_{ab}^-}F_u\right)+\pmat{\bf 0&\bf 0\\ \bf 0 &*} \\
&= \frac12 \pmat{1 &1 & 0 & \hdots & 0 \\1 & 1 & 0 & \hdots & 0 \\ 0 & 0 & * & \hdots & * \\ \vdots & \vdots & \vdots & \ddots & \vdots \\ 0 & 0 & * & \hdots & *} +
\frac{e^{-\ii \tau \mu_s}}{2}\pmat{1 & -1 & 0 & \hdots & 0 \\-1 & 1 & 0 & \hdots & 0 \\ 0 & 0 & * & \hdots & * \\ \vdots & \vdots & \vdots & \ddots & \vdots \\ 0 & 0 & * & \hdots & *}.
\end{align*}
Hence we have proper LaFR at time $\tau$.

For the only if direction, suppose proper LaFR occurs between $a$ and $b$. Condition (i) follows from Theorem \ref{LaFRSC}. Using Corollary \ref{ratiocondition}, Theorem \ref{algclosed} and Theorem \ref{harmonious}, we have integers $a^+$, $a^-$, and square free integers $\Delta^+$, $\Delta^-$, and $\{b_r\}_{r=0}^d$, such that, for all $r = 0,..,d$,
\begin{enumerate}[(1)]
\item if $\mu_r \in \Phi_{ab}^+$, then $\displaystyle \mu_r = \frac{a^+ + b_r \sqrt{\Delta^+}}{2}$, and
\item if $\mu_r \in \Phi_{ab}^-$, then $\displaystyle  \mu_r = \frac{a^- + b_r \sqrt{\Delta^-}}{2}$.
\end{enumerate}
Define 
\[g^+ = \gcd\left\{\frac{\mu_r - \mu_s}{\sqrt{\Delta^+}}: {\mu_r,\mu_s \in \Phi_{ab}^+}\right\}\]
and
\[g^- = \gcd\left\{\frac{\mu_r - \mu_s}{\sqrt{\Delta^-}}: {\mu_r,\mu_s \in \Phi_{ab}^-}\right\}.\]
Then $\tau$ is an integer multiple of $2\pi/g^+\sqrt{\Delta^+}$ and also an integer multiple of $2\pi/g^-\sqrt{\Delta^-}$. Therefore $\Delta^+=\Delta^-$. Let $\Delta=\Delta^+$.

We now prove Condition (ii). By Corollary \ref{pmsetsize}, $\Phi_{ab}^+$ contains at least two elements. Since $0\in \Phi_{a,b}^+$, we must have $a^+=0$. As $\Phi_{ab}^+$ is closed under taking algebraic conjugates, if $\Delta\ne 1$, then for any
\[\frac{0+b_r\sqrt{\Delta}}{2}\in \Phi_{a,b}^+,\]
we have
\[\frac{0-b_r\sqrt{\Delta}}{2}\in \Phi_{a,b}^+.\]
However, this contradicts the fact $L$ is positive semidefinite, as all eigenvalues of $L$ should be non-negative. Therefore $\Delta=1$, and so $\Phi_{a,b}^+\cup \Phi_{a,b}^-\subseteq \ZZ^{\ge 0}$.

Finally, by Lemma \ref{constant}, LaFR can only occur at at times $\frac{2\pi}{g}\ZZ$, where $g$ is defined as in this theorem. If $\mu_r\mapsto e^{i\tau\mu_r}$ is constant on the entire set $\Phi_{ab}^+\cup\Phi_{ab}^-$, then the $X$ is periodic with respect to $L$. Hence, there must be some eigenvalue in $\Phi_{a,b}^-$ that is not divisible by $g$.
\end{proof}

A direct consequence of this result is that the vertices involved in proper LaFR must have degree at least two, unless the graph has fewer than five vertices.

\begin{corollary} \label{LaFRdegbound}
On a connected graph with $n\ge 5$ vertices, if proper LaFR occurs between vertices $a,b$, then $deg(a)=deg(b)\ge 2$.
\end{corollary}

\begin{proof}
Since $a$ and $b$ are Laplacian strongly cospectral, by Theorem \ref{scdegbound}, we have
\[\frac{n-2}{n}\lambda^+ + \sigma \le d,\]
where $\lambda^+$ is the smallest element in $\Phi_{ab}^+$, and $\sigma=1$ if $a$ is adjacent to $b$ and $\sigma=0$ otherwise. As the LaFR is proper, $g\ne 1$, and so $\lambda\ge 2$. The result then follows when $n\ge 5$.
\end{proof}

\section{Laplacian periodicity}
\label{Section:Period}
In this section, we characterize Laplacian periodicity at a vertex $a$ of a graph. This can be viewed as non-proper LaFR; however, it is also a necessary condition for proper LaFR to occur from $a$ to another vertex.

As before, let $X$ be a graph with Laplacian matrix $L$, and suppose the spectral decomposition is
\[L = \sum_r \mu_r F_r.\]
The \textsl{eigenvalue support} of a vertex $a$, denoted $\Phi_a$, is the set
\[\Phi_a = \{\mu_r: F_r e_a\ne 0 \}.\]
Since $F_r$ is positive semidefinite, equivalently,
\[\Phi_a = \{\mu_r: (F_r)_{aa}\ne 0.\}\]
We will let $\Phi_a^0$ denote the complement of $\Phi_a$.

Define two polynomials
\[\psi(t)=\det(tI-L)\]
and
\[\psi_a(t)=\det((tI-L)[a|a]).\]
We can compute $\Phi_a$ using $\psi$ and $\psi_a$. By Cramer's Rule and the spectral decomposition of $L$, 
\[\frac{\psi_a(t)}{\psi(t)}=(tI-L)^{-1}_{aa}=\sum_r \frac{1}{t-\mu_r} F_r.\]
This leads to the following observation.

\begin{lemma}\label{lem:poles}
The elements in $\Phi_a$ are precisely the poles of $\psi_a(t)/\psi(t)$.
\end{lemma}

We also note an interesting connection between $\Phi_a$ and the spanning trees of $X$. Suppose $X$ has $n$ vertices and $\deg(a)=d$. Let $\mu_r$ be an eigenvalue of $L$ with eigenvector $x$. If $x_a=0$, then we may put the Laplacian matrix in the form
\[L=\pmat{
L[a\vert a]& * \\
* & d },\]
where the last row is indexed by $a$. Then the restriction of $x$ to $X\setminus a$, denoted $\hat{x}$, is an eigenvector of $L[a|a]$ with eigenvalue $\mu_r$. Hence, the restriction of vectors in the $\mu_r$-eigenspace to $X\backslash a$, if non-zero, is an eigenvector for $L[a|a]$ with eigenvalue $\mu_r$.

On the other hand, if we let $W$ denote the subspace $\RR^{n-1}\times \{0\}$, then $\mu\in \Phi_a^0$ if and only if the $\mu_r$-eigenspace is a subspace of $W$, and so
\begin{align*}
    \dim (\col(F_r)\cap W)&=\dim (\col(F_r))+\dim W-\dim (\col(F_r)+W)\\
    &=\begin{cases}
    \dim (\col(F_r)),& \text{if }   \mu\in\Phi_a^0
    \\\dim(\col(F_r))-1,& \text{if } \mu\in\Phi_a
    \end{cases}
\end{align*}
Therefore, if $\mu_r\in\Phi_a^0$, then $\mu_r$ is an eigenvalue of $L[a|a]$ of same or higher multiplicity, and if $\mu_r\in \Phi_a$, then $\mu_r$ is an eigenvalue of $L[a|a]$ with multiplicity at least one less. By the Matrix-Tree theorem, we arrive at the following.

\begin{theorem}\label{lem:sptrees}
Let $X$ be a graph and let $a$ be a vertex. The product of elements in $\Phi_a^0$ divides the number of spanning trees of $X$. 
\end{theorem}

Theorem \ref{CharThm} gives a characterization of proper LaFR between $a$ and $b$ using $\Phi_{ab}^+$ and $\Phi_{ab}^-$. We now use $\Phi_a$ to characterize Laplacian periodicity. The proof is very similar, so we omit it here.

\begin{theorem}{(Characterization Theorem For Laplacian Periodicity)}\label{LaPerCharThm}
A graph is Laplacian periodic at a vertex $a$ if and only if  $\Phi_a$ contains only integers. Moreover, if Laplacian periodicity occurs at time $\tau$, then $\tau$ is an integer multiple of $2\pi/G$, where 
\[G=\gcd\Phi_a.\]
\end{theorem}

Clearly, if $a$ is strongly cospectral to $b$, then 
\[\Phi_a = \Phi_{ab}^+ \cup \Phi_{ab}^-,\]
and 
\[\Phi_a^0 = \Phi_{ab}^0.\]
Thus we have the following corollary.

\begin{corollary}\label{cor:FRtoper}
If $X$ admits proper LaFR between $a$ and $b$, then it is periodic at both $a$ and $b$.
\end{corollary}

In \cite{g12b}, Godsil showed that for any integer $k$, there are only finitely many connected graphs with maximum valency that admit adjacency perfect state transfer. We extend his result to Laplacian periodicity. 

Given a graph $X$ and a vertex $a$, the \textsl{eccentricity} of $a$, denoted $\ecc(a)$, is the maximum distance from $a$ to any other vertex in $X$. We show that the size of $\Phi_a$ determines an upper bound for $\ecc(a)$.

\begin{lemma}\label{eccentricity}
Let $\Phi_a$ be the eigenvalue support of $a$ with respect to the Laplacian matrix. Then $\abs{\Phi_a}\ge \ecc(a)+1$.
\end{lemma}

\begin{proof}
Let $\ell=\ecc(a)$. Let $d$ be the maximum valency of the graph, and define
\[M=dI-L.\]
Then $M$ is a weighted adjacency matrix of the graph with non-negative entries. Thus, for any $i<j$, the support of $M^i e_a$ is a proper subset of the support of $M^je_a$, and so the vectors
\[e_a, Me_a, M^2e_a,\cdots,M^{\ell} e_a\]
are linearly independent. On the other hand, $M$ is a linear combination of the spectral idempotents $F_r$ of $L$, so
\[\mathrm{span}\{e_a, Me_a, M^2e_a,\cdots,M^{\ell} e_a\}\subseteq\mathrm{span}\{F_r e_a: r\in\Phi_a\}.\]
Thus, $1+\ell\le \abs{\Phi_a}$.
\end{proof}

Using the second upper bound in Theorem \ref{Levbd}, we are able to prove the following theorem. It explains why Laplacian periodicity, and hence LaFR, is a rare phenomenon.

\begin{theorem}
Given an integer $k$, there are only finitely many connected graphs with maximum valency at most $k$ that are Laplacian periodic at a vertex.
\end{theorem}
\begin{proof}
Let $X$ be a graph with maximum valency $k$. Suppose $X$ is Laplacian periodic at a vertex $a$. Then $\Phi_a$ contains only integers. By Theorem \ref{Levbd}, the eigenvalues of $L$ are no greater than $2k$. Hence, 
\[2k+1\ge \abs{\Phi_a} \ge \ecc(a)+1,\]
from which we have $\ecc(a)\le 2k$, and only finitely many graphs satisfy this constraint.
\end{proof}

\section{Laplacian fractional revival is polynomial time}
Using the theory we developed so far, we show that LaFR can be decided in polynomial time.

\begin{theorem}\label{thm:polytime}
Deciding whether a graph has proper LaFR, and the earliest time when it occurs, can be done in polynomial time.
\end{theorem}
\begin{proof}
We prove this by showing that all three conditions in Theorem \ref{CharThm} can be checked in polynomial time. 

By lemma 2.4 in \cite{PolyTime}, Laplacian strong cospectrality, which is Condition (i), can be checked in polynomial time. 

Now we adapt lemma 2.5 in \cite{PolyTime} to proper LaFR. By Lemma \ref{lem:poles}, the elements in $\Phi_a$ are precisely the poles of $\psi_a(t)/\psi(t)$, which are all simple. Equivalently, $\Phi_a$ consists of simple roots of
\begin{align*}
    f(x) = \frac{\psi(x)}{gcd(\psi(x), \psi_a(x))}
\end{align*}
To find integer eigenvalues in $\Phi_a$, recall that all eigenvalues of $L$ lie in $[0,n]$, so we simply check whether $0,1,\cdots, n$ are roots of $f(x)$, which can be done in polynomial time. Moreover, if $f(x)$ has degree $k$, then the coefficient of $(-x)^{k-1}$ is the sum of the roots of $f(x)$. By comparing the sum of the roots we found with the coefficient of $(-x)^{k-1}$, we can decide whether all eigenvalues in $\Phi_a$ are integers. 

We then use Gaussian elimination to calculate the corresponding eigenvectors in polynomial time. This means we can decide in polynomial time how to partition $\Phi_a$ into $\Phi_{ab}^+$ and $\Phi_{ab}^-$. Therefore, Condition (ii) can be checked in polynomial time.

Finally, let
\[g = \gcd \left\{\mu_r-\mu_s:\mu_r, \mu_s \in \Phi_{ab}^+ \text{ or } \mu_r, \mu_s\in\Phi_{ab}^- \right\}.\]
Since there are at most $n$ eigenvalues in $\Phi_{ab}^+\cup\Phi_{ab}^-$, the set over which we take the gcd has size at most $\binom{n}{2} = O(n^2)$. Moreover, as all elements in the set live in $[0,n]$, we can compute $g$ in polynomial time. It remains to check no element in $\Phi_{ab}^-$, which has size less than $n$, is divisible by $g$. Hence, Condition (iii) can be checked in polynomial time. 
\end{proof}

\section{No proper Laplacian fractional revival on trees} \label{sec:trees}
In this section, we show that proper LaFR does not occur on trees except for $K_2$ and $P_3$. An earlier result about LaPST on trees, due to Coutinho and Liu \cite{LaPST}, can be viewed as a consequence.

We first prove a technical lemma on the signed incidence matrices of trees.

\begin{lemma}\label{lem:By}
Let $T$ be a tree on $n$ vertices. Fix an orientation of $T$, and let $B$ be the signed incidence matrix. If $\mu\ge 2$ is an integer, then any solution to
\[By\equiv 0\pmod{\mu}\]
must satisfy 
\[y\equiv 0\pmod{\mu}.\]
\end{lemma}
\begin{proof}
If $T=K_2$, then 
\[B=\pmat{1\\-1},\]
and from $By\equiv 0\pmod{\mu}$ it follows that both entries of $y$ are divisible by $\mu$.

Now let $T$ be a general tree on $n\ge 3$ vertices. Let $v$ be a leaf of the tree, and $e$ the edge incidence to $v$. Then $By\equiv 0\pmod{\mu}$ implies that $y_e$ is divisible by $\mu$. Thus, if $y'$ denotes the restriction of $y$ to $T\backslash v$, then
\[B[v|e] y'\equiv 0\pmod{\mu}.\]
Note that $B[v|e]$ is the signed incidence matrix of some orientation of $T\backslash v$. By induction, we see that $y\equiv 0\pmod{\mu}$.
\end{proof}

As $L=BB^T$, the above lemma imposes number theoretic conditions on the eigenvectors of $L$ associated with integer eigenvalues.

\begin{corollary}\label{cor:diffdiv}
Let $T$ be a tree with Laplacian matrix $L$. Let $\mu$ be an integer eigenvalue of $L$ with eigenvector $x$. Suppose the entries of $x$ are integers. Then the difference of any two entries of $x$ is divisible by $\mu$.
\end{corollary}
\begin{proof}
Let $B$ be the signed incidence matrix of any orientation of $T$. Then $L=BB^T$. Define a new vector $y=B^Tx$. The condition $Lx=\mu x$ is equivalent to $By=\mu x$, and so
\[By\equiv 0{\pmod \mu}.\]
By Lemma \ref{lem:By}, all entries of $y$ are divisible by $\mu$. 

On the other hand, for every edge $e=(u,v)$ of $X$, we have $y_e=x_u-x_v$. Thus $x_u-x_v$ is divisible by $\mu$. As $T$ is connected, the difference of any two entries of $x$ are divisible by $\mu$.
\end{proof}

We now prove the main non-existence result about proper LaFR on trees.
\begin{theorem}\label{thm:notree}
The only trees the admit proper LaFR are $K_2$ and $P_3$.
\end{theorem}
\begin{proof}
Let $T$ be a tree with proper LaFR between $a$ and $b$. If $\abs{\Phi_{ab}^-}=1$, then $a$ and $b$ are twin vertices by Corollary \ref{scdist}. As $T$ is a tree, there can only be one common neighbor of $a$ and $b$, and so they are leaves. Thus by Corollary \ref{LaFRdegbound}, $T$ has at most four vertices. It is easy to check that only $K_2$ and $P_3$ admit LaFR.

Now suppose $\abs{\Phi_{ab}^-}\ge 2$. Let $\mu\in \Phi_{ab}^-$ be an eigenvalue of $L$ with eigenvector $x$. Since $\mu$ is an integer, we may assume without loss of generality that $x$ has integer entries, and that the gcd of these entries is $1$. By Corollary \ref{cor:diffdiv},
\[x_u\equiv x_v \pmod{\mu}.\]
However, as $\mu\in \Phi_{ab}^-$, we also have $x_u=-x_v$. Hence $\mu$ divides $2x_u$. Moreover, if $\gcd(\mu, x_u)$ were not $1$, then it would appear as a common factor of all entries of $x$, which contradicts our assumption. Therefore $\mu$ divides $2$. Since $\abs{\Phi_{ab}^-}\ge 2$, we must have $\Phi_{ab}^-=\{1,2\}$. By Theorem \ref{CharThm}, LaFR occurs at time $2\pi/g$ with $g=1$, and this cannot be proper.
\end{proof}

\section{Laplacian fractional revival on joins}
\label{Section:Joins}
We derive more results about proper LaFR on join graphs, in addition to those in Section \ref{sec:compjoin}. In particular, we determine when and where LaFR can occur on a join graph.

\begin{theorem}\label{malenaconjecturejoin}
Let $Z$ be a join graph on $n\ge 3$ vertices. Suppose $Z$ admits LaFR at time $\tau$. Then $\tau$ is an integer multiple of $2\pi/n$.
\end{theorem}
\begin{proof}
Let $L$ be the Laplacian matrix of $Z$, and $\comp{L}$ the Laplacian matrix of $\comp{Z}$. Recall from the proof of Theorem \ref{thm:complement} that
\begin{align*}
\exp(i\tau\comp{L})
&=e^{in\tau}\left((e^{-in\tau}-1)\frac{1}{n}J+I\right)\exp(-i\tau L)\\
&=\frac{1-e^{in\tau}}{n}J + e^{in\tau} \exp\left(-i\tau L\right)
\end{align*}
Assume, for a contradiction, that $\tau$ is not an integer multiple of $2\pi/n$. Then in the above formula, the coefficient before $J$ is non-zero. Since $\comp{X}$ has at least three vertices, this can only happen if $\comp{X}$ is connected. However, that contradicts the fact that $X$ is a join graph.
\end{proof}

We saw in Theorem \ref{thm:complement} that, under a mild conditon, LaFR is preserved by taking the complement of the graph. As a direct consequence of Theorem \ref{malenaconjecturejoin}, we can drop this condition when the graph is a join.

\begin{corollary}\label{cor:joincomplement}
Let $Z$ be a connected join graph. Then $Z$ admits LaFR between $a$ and $b$ if and only if $\comp{Z}$ admits LaFR between $a$ and $b$ at the same time.
\end{corollary}

The following result determines where LaFR can occur on a join graph. 

\begin{corollary}\label{samepartition}
Let $Z=X+Y$ be a join graph on at least three vertices. Suppose proper LaFR occurs between $a$ and $b$ at time $\tau$. Then $a$ and $b$ are either both in $X$ or both in $Y$. Moreover, the component of $\comp{Z}$ containing $a$ admits proper LaFR between $a$ and $b$ at time $\tau$. 
\end{corollary}
\begin{proof}
By Corollary \ref{cor:joincomplement}, the complement of $Z$ admits LaFR between $a$ and $b$ at time $\tau$. Hence $a$ and $b$ are in the same component of $\comp{Z}$. It follows from the definition of a join graph that $a$ and $b$ are either both in $X$ or both in $Y$. The second statement is a consequence of Corollary \ref{cor:joincomplement}.
\end{proof}

Our theory enables us to build infinite families of join graphs that admit proper LaFR. Unlike those constructed in Theorem \ref{BasicDBCone}, these graphs do not have to be double cones. To start, we cite a spectral result from \cite{PSTLa}.

{\begin{lemma}\cite{PSTLa}
\label{lem:joines}
Let $X$ and $Y$ be two graphs on $m$ and $n$ vertices, respectively. Then $m+n$ is an eigenvalue of $X+Y$ with an eigenvector being
\[ \pmat{n\one_m\\-m\one_n}.\]
\end{lemma}}

\ignore{\begin{lemma}\cite{PSTLa}
\label{lem:joines}
Let $X$ and $Y$ be two graphs on $m$ and $n$ vertices, respectively. Let the spectral decompositions of their Laplacian matrices be
\[L(X)=\sum_r \lambda_r E_r,\quad L(Y)=\sum_r\theta_r F_r.\]
Then the spectral decomposition of $L(X+Y)$ is
\[L(X+Y)=0\cdot z_0z_0^T + (m+n)z_1z_1^T +\sum_{r: \lambda_r\ne m} (n+\lambda_r)\pmat{E_r & 0 \\0&0} + \sum_{s:\theta_s\ne n} (m+\theta_s)\pmat{0&0\\0&F_s},\]
where $z_0$ and $z_1$ are vectors given by
\[z_0 = \frac{1}{\sqrt{m+n}} \one_{m+n},\quad z_1=\frac{1}{\sqrt{mn(m+n)}} \pmat{n\one_m\\-m\one_n}.\]
\end{lemma}}

We remark an immediate consequence of the structure of the eigenvectors of a join.

\begin{corollary}\label{cor:ninplus}
Let $X+Y$ be a join graph on $\ell$ vertices, and let $a$ and $b$ be two Laplacian strongly cospectral vertices.
\begin{enumerate}[(i)]
    \item If $a$ and $b$ both lie in $X$ or both lie in $Y$, then $\ell\in\Phi_{ab}^+$.
    \item If $a$ lies in $X$ and $b$ lies in $Y$, then $\ell\in\Phi_{ab}^-$.
\end{enumerate}
\end{corollary}

We now show that any graph with proper LaFR at a special time yields an infinite family of graphs with LaFR.

\begin{theorem}\label{thm:infjoin}
Let $X$ be a graph on $m\ge 3$ vertices. Suppose for some factor $g$ of $m$, proper LaFR occurs on $X$ bewteen $a$ and $b$ at time $2\pi/g$. Let $Y$ be any graph on $n$ vertices such that $g$ divides $n$. Then the join graph $X+Y$ admits proper LaFR between $a$ and $b$ at time $2\pi/g$.
\end{theorem}
\begin{proof}
Since $g$ divides $m$, by Theorem \ref{thm:complement}, there is proper LaFR at time $2\pi/g$ between $a$ and $b$ on $\comp{X}$, and thus on $\comp{X}\cup \comp{Y}$. As $g$ also divides $n$, applying Theorem \ref{thm:complement} yields the result.
\end{proof}

The rest of this section is devoted to LaFR on double cones. In Theorem \ref{BasicDBCone}, we showed that every double cone on $n\ge 4$ vertices have proper LaFR at time $2\pi/n$. Here, we prove that the converse is also true. Our result uses the following simple observation.

\begin{lemma}
Let $X$ be a graph on $n$ vertices. If $n$ is an eigenvalue of $L(X)$, then $X$ is a join graph.
\end{lemma}
\begin{proof}
By Theorem \ref{Levbd}, $\comp{X}$ has at least two components. 
\end{proof}

With this, we are able to prove the converse of Theorem \ref{BasicDBCone}.

\begin{theorem}\label{DCTime}
Let $Z$ be a graph on $n\ge 3$ vertices that admit proper LaFR at time $2\pi/n$. Then $Z$ is a double cone.
\end{theorem}
\begin{proof}
Define 
\[g = \gcd \left\{\mu_r-\mu_s:\mu_r, \mu_s \in \Phi_{ab}^+ \text{ or } \mu_r, \mu_s\in\Phi_{ab}^- \right\}.\]
By Theorem \ref{Levbd}, the eigenvalues $L(Z)$ are no greater than $n$, and as $0\in\Phi_{ab}^+$, we must have $g\le n$. On the other hand, Theorem \ref{CharThm} says that proper LaFR must occur at times that are integer multiples of $2\pi/g$. Hence $g=n$. Now, as $\Phi_{ab}^+$ contains at least two elements, it can only be that $\Phi_{ab}^+=\{0,n\}$. Therefore, $Z$ is a join graph.

By Corollary \ref{samepartition}, there is an induced subgraph $X$ of $Z$ containing $a$ and $b$ such that
\begin{enumerate}[(i)]
    \item $Z=X+Y$ for some induced subgraph $Y$ of $Z$, and
    \item $\comp{X}$ admits proper LaFR between $a$ and $b$ at time $2\pi/n$.
\end{enumerate}
We claim that $X$ has no other vertices than $a$ and $b$. Suppose otherwise, and let $m$ be the number of vertices in $X$. Since $m\ge 3$, by Theorem \ref{CharThm}, the earliest time when proper LaFR could occur on $\comp{X}$ is $2\pi/m$. However, as $m<n$, this contradicts the fact that $\comp{X}$ admits proper LaFR at time $2\pi/n$. Therefore, $m=2$ and $Z$ is indeed a double cone.
\end{proof}

We reach the same conclusion if the graph with proper LaFR has an prime number of vertices.

\begin{theorem}\label{thm:primedbcone}
Let $p$ be an odd prime. Let $X$ be a graph on $p$ vertices. Suppose proper LaFR occurs on $X$. Then $X$ is a double cone.
\end{theorem}
\begin{proof}
Suppose $X$ admits proper LaFR between vertices $a$ and $b$. Let $q$ be the number of spanning trees of $X$. By Lemma \ref{lem:sptrees}, $q$ is divisible by the product of elements in $\Phi_a^0$. On the other hand, the Matrix-Tree Theorem states that $pq$ equals the product of non-zero eigenvalues of $L(X)$. Hence $p$ divides
\[\prod_r \Phi_a \backslash\{0\}.\]
As $p$ is a prime, it lies in $\Phi_a$ and in particular, it is an eigenvalue of $L(X)$. Therefore, $X$ is a join graph.

By Theorem \ref{CharThm}, the earliest time when proper LaFR occurs is $2\pi/g$, where
\[g = \gcd \left\{\mu_r-\mu_s:\mu_r, \mu_s \in \Phi_{ab}^+ \text{ or } \mu_r, \mu_s\in\Phi_{ab}^- \right\}.\]
From Corollary \ref{cor:ninplus} we see that $p\in\Phi_{ab}^+$, and so $g$ divides $p$. However $g\neq 1$. It follows that $g=p$, and by Theorem \ref{DCTime}, $X$ is a double cone.
\end{proof}

We noticed, from the data on small graphs, that most examples with proper LaFR are double cones. This is partially explained by the above two results.

\section{Polygamy of proper fractional revival}
It is known that perfect state transfer, with respect to either the Laplacian matrix or the adjacency matrix, is \textsl{monogamous}: if a graph admits perfect state transfer from $a$ to $b$ and from $a$ to $c$, then $b=c$. In contrast, we show that LaFR can be \textsl{polygamous}: a vertex may be paired with two different vertices for proper LaFR. The following trick enables us to construct an infinite family of examples.

\begin{lemma}\label{lem:polygamy}
Let $X$ and $Y$ be two graphs on at least three vertices. Suppose $X$ admits proper LaFR from $a$ to $b$ at time $2\pi/g$, and $Y$ admits proper LaFR from $c$ to $d$ at time $2\pi/h$, where $g$ and $h$ are the gcds defined in Theorem \ref{CharThm}. Let 
\[G=\gcd\Phi_a,\quad H=\gcd\Phi_c.\]
Further assume that $G$ is not divisible by $\gcd(g,H)$, and $H$ is not divisible by $\gcd(h,G)$. Then $X\cart Y$ admits LaFR from $(a,c)$ to $(b,c)$, from $(a,d)$ to $(b,d)$, from $(a,c)$ to $(a,d)$, and from $(b,c)$ to $(b,d$).
\end{lemma}
\begin{proof}
By Theorem \ref{CharThm} and Theorem \ref{LaPerCharThm}, the following phenomena occur.
\begin{enumerate}[(i)]
    \item At time $2\pi/\gcd(g,H)$, $X$ has proper LaFR between $a$ and $b$, and $Y$ is Laplacian periodic at $c$ and $d$.
    \item At time $2\pi/\gcd(h,G)$, $X$ is Laplacian periodic at $a$ and $b$, and $Y$ has proper LaFR between $c$ and $d$.
\end{enumerate}
The result now follows from Theorem \ref{thm:cartproduct}.
\end{proof}

We use this lemma to build an infinite family of graphs on which LaFR is polygamous. Our construction involves two types of distance regular graphs: Hadamard graphs, and distance regular double covers of complete graphs. The eigenvalues of both types of graphs can be found in the table of \cite[Sec 3]{FRschemes}. Since these graphs are regular, fractional revival with respect to the Laplacian matrix is equivalent to fractional revival with respect to the adjacency matrix.

Given an $n\times n$ Hadamard matrix $H$, we define $4n$ symbols $r_i^+, r_i^-, c_i^+, c_i^-$, where $i=1,2,\cdots, n$. The \textsl{Hadamard graph} is a graph with these symbols as vertices, such that $r_i^{\pm}$ is adjacent to $r_j^{\pm}$ if $H_{ij}=1$, and $r_i^{\pm}$ is adjacent to $r_j^{\mp}$ if $H_{ij}=-1$. It is shown in \cite{FRschemes} that the Hadamard graph has proper fractional revival between antipodal vertex $a$ and $b$ with
\[\Phi_{ab}^+=\{0,n^2,2n^2\},\quad \Phi_{ab}^-=\{n^2-n,n^2+n\}.\]

In \cite{FRschemes}, the author also characterized fractional revival on distance regular double cover of $K_m$. If $m=4p^2$ for some odd prime $p$ and $\delta=2$, then the double cover has proper LaFR between the antipodal vertices $c$ and $d$ with
\[\Phi_{cd}^+=\{0,4p^2\},\quad \Phi_{cd}^-=\{4p^2-2-2p,\, 4p^2-2+2p\}.\]

It is known that there exist distance regular antipodal covers of $K_{36}$ with $\delta=2$. Thus we have the following.

\begin{theorem}\label{thm:polygamy}
Let $n=6q$ for some odd positive integer $q$. Let $X$ be a Hadamard graph on $4n^2$ vertices. Let $Y$ be a distance regular double cover of $K_{36}$ with $\delta=2$. If $a$ and $b$ are antipodal vertices in $X$, and $c$ and $d$ are antipodal vertices in $Y$, then the following occur.
\begin{enumerate}[(i)]
    \item $X\cart Y$ admits LaPST from $(a,c)$ to $(b,c)$ and from $(a,d)$ to $(b,d)$ at time $\pi/2$.
    \item $X\cart Y$ admits proper LaFR from $(a,c)$ to $(a,d)$ and from $(b,c)$ to $(b,d)$ at time $\pi/3$.
\end{enumerate}
\end{theorem}
\begin{proof}
Define
\begin{align*}
    g &= \gcd \left\{\mu_r-\mu_s:\mu_r, \mu_s \in \Phi_{ab}^+ \text{ or } \mu_r, \mu_s\in\Phi_{ab}^- \right\}\\
    h &= \gcd \left\{\mu_r-\mu_s:\mu_r, \mu_s \in \Phi_{cd}^+ \text{ or } \mu_r, \mu_s\in\Phi_{cd}^- \right\}
\end{align*}
and 
\[G=\gcd\Phi_a,\quad H=\gcd\Phi_c.\]
Then 
\[g=12q,\quad h=12,\quad G=6q,\quad H=4.\]
Thus $\gcd(g,H)$ does not divide $G$, and $\gcd(h, G)$ does not divide $H$. Now apply Lemma \ref{lem:polygamy}.
\end{proof}

We remark that, as these graphs are regular, they are also the first infinite family of unweighted graphs that admit polygamous adjacency fractional revival. This answers an open question in \cite{FRgraphs}.

\section*{Acknowledgement}
This project was completed under the Fields Undergraduate Summer Research Program.  Zhan acknowledges the support of the York Science Fellow program.
The authors would like to thank Gabriel Coutinho, Chris Godsil and Christino Tamon for useful discussions.



\end{document}